\documentclass[12pt]{article}
\usepackage{amsmath}
\usepackage{amssymb}
\usepackage{amsthm}

\title{The ideals of the homological Goldman Lie algebra}
\author{Kazuki Toda
\thanks{Affiliation:Graduate School of Mathematical Sciences, University of Tokyo 
3-8-1 Komaba Meguro-ku, Tokyo 153-8914, JAPAN \ 
Email:ktoda@ms.u-tokyo.ac.jp}
}
\date{\today}

\theoremstyle{definition}
\newtheorem{Def}{Definition}[section]
\newtheorem{rem}[Def]{Remark}
\newtheorem{ex}[Def]{Example}

\theoremstyle{plain}
\newtheorem{Prop}[Def]{Propositon}
\newtheorem{Lem}[Def]{Lemma}
\newtheorem{Thm}[Def]{Theorem}
\newtheorem{Cor}[Def]{Corollary}

\def\mbN{\mathbb{N}}
\def\mbZ{\mathbb{Z}}
\def\mbQ{\mathbb{Q}}
\def\mbR{\mathbb{R}}

\begin{document}
\maketitle

\begin{abstract}
   We determine all the ideals of the homological Goldman Lie algebra,  which reflects the structure of an oriented surface. 
\end{abstract} 
\section{Introduction}
By a {\it surface}, we call an oriented compact connected two-dimensional smooth manifold with boundary. 
It is well known that the first homology group and the intersection form of a surface reflect the topological structure of the surface. For example, they have information about the genus and the boundary components of the surface. 

To study them in details, we consider a Lie algebra coming from them. We call it the homological Goldman Lie algebra. 
Goldman introduced the Lie algebra for study of the moduli space of $GL_1(\mbR )$-flat bundles over the surface[G]p.295-p.297. We will define the Lie algebra later. 

The homological Goldman Lie algebra is infinite dimensional and we can define this algebra only from algebraic information. 
So, it is interesting in an algebraic context. 
The homological Goldman Lie algebra comes from the first homology group and its intersection form. 
So, it is also interesting in a geometric context. 
For example, the homological Goldman Lie algebra is exactly the subalgebra of all Fourier polynomials in the Poisson algebra on the symplectic torus if the surface is closed. 
Moreover, we can consider a more complicated Lie algebra coming from free loops and the intersection form. 
Take two free loops $\alpha$ and $\beta$ on the surface in general position. 
We define $[\alpha ,\beta ]:=\sum _{p \in \alpha \cap \beta}\varepsilon (p;\alpha ,\beta )\alpha \cdot _p \beta$, where $\varepsilon (p;\alpha ,\beta )$ is the local intersection number of $\alpha$ and $\beta$ at $p$, and $\alpha \cdot _p\beta$ is the free homotopy class of the product in the fundamental group with base point $p$. 
The bracket induces a well-defined operator in the free module with basis the set of free loops, and it is easy to show that the bracket is skew and satisfies the Jacobi identity. 
We call this Lie algebra the Goldman Lie algebra. 
We have a surjective Lie algebra homomorphism from the Goldman Lie algebra onto the homological Goldman Lie algebra[G]p.295-p.297. 

This algebra reflects the topological structure of the surface. 
The purpose of this paper is to study the algebraic structure of the homological Goldman Lie algebra. 
More precisely, we determine the ideals of the homological Goldman Lie algebra. 
The complexity of the Lie algebra depends on the number of the boundary components of the surface. 
If the surface is closed, the number of all the ideals of the homological Goldman Lie algebra is finite. 


Now, we will define the homological Goldman Lie algebra, and formulate our result. 
Let $H$ be a $\mbZ$-module, i.e., an abelian group, and $\langle-,-\rangle :H\times H \rightarrow \mbZ ;(x,y)\mapsto \langle x,y \rangle$ an alternating $\mbZ$-bilinear form. 
For example, we consider $H$ is the first homology group, and $\langle -,-\rangle$ the intersection form on $H$ of a surface. 
We define a $\mbZ$-linear map
$\mu :H\rightarrow {\rm Hom} _{\mbZ}(H,\mbZ )$
by 
$\mu (x)(y)=\langle x,y\rangle$. 
Denote by $\mbQ H$ the $\mbQ$-vector space with basis the set $H$;
\[
\mbQ H :=\left\{ \sum _{i=1}^n c_i[x_i] \mid n \in \mbN ,c_i \in \mbQ ,x_i \in H \right\} ,
\]
where $[-]:H\rightarrow \mbQ H$ is the embedding as basis. 
Here, we remark $c[x]\neq [cx]$ for any $c\neq 1$ and $x\in H$. 
For $x,y\in H$, we define a bracket $[-,-]:\mbQ H\times \mbQ H\rightarrow \mbQ H$ by 
$[[x],[y]]:=\langle x,y\rangle [x+y]$. 
It is easy to see that this bracket is skew and satisfies the Jacobi identity \cite{G}p.295-p.297. 
The Lie algebra $(\mbQ H,[-,-])$ is called the {\it homological Goldman Lie algebra of $(H,\langle -,-\rangle )$}. 

For any subgroup $A$ of $H$, 
the Lie algebra $\mbQ H$ is a graded Lie algebra of type $H/A$ \cite{B}Chapter I\hspace{-.1em}I \S 11. Namely, we have
$\mbQ H = \bigoplus _{x+A \in H/A} \mbQ (x+A)$ and $[\mbQ (x+A),\mbQ (y+A)]\subset \mbQ (x+y+A)$. 
Here, $\mbQ (x+A)=\{ \sum _{i=1}^nc_i[x_i]\in \mbQ H \mid x_i \in x+A\}$. 
The {\it $A$-degree} is the degree induced by grading of type $H/A$. 
For $X \in \mbQ (x+A)$, we denote $A${\rm -deg} $X = x+A$. 
In particular, The Lie algebra $\mbQ H$ is graded by the group $\bar{H}:=H/{\rm {\rm ker}}\mu$. 
For $x\in H$, we denote $\bar{x} := x+{\rm {\rm ker}}\mu$.

For $x \in H$, we define $T(x):\mbQ H\rightarrow \mbQ H$ by 
$T(x)([y])=[x+y]$, where $y\in H$. 
The map $T:H \rightarrow GL(\mbQ H)$ induces an action of $H$ on $\mbQ H$. 
For $x,y \in H$ and $Y \in \mbQ \bar{y}$, we have 
$[[x],Y]=\langle x,y\rangle T(x)(Y)$. 

Let $A$ be a subgroup of $H$, and $V$ a subspace of $\mbQ H$. 
We say that {\it $V$ has the $A$-stability} if $T(A)(V)\subset V$. 

Our main theorem is the following. 
\begin{Thm}[Classification of the ideals of $\mbQ H$]\label{The ideals}
For any ideal $\mathfrak{h}$ of $\mbQ H$, there exists a unique pair $(V_0,V)$ such that

(1) $V_0$ and $V$ are subspaces of $\mbQ \bar{0}=\mbQ {\rm ker}\mu$, 

(2) $V$ has the ${\rm ker}\mu$-stability, and 

(3) $\mathfrak{h} = V_0 \oplus \bigoplus _{\bar{x} \in \bar{H}\setminus \bar{0}}T(\bar{x})(V)$,

where $T(\bar{x})(V) := T(x)(V)$, which is well-defined by (2). 
If $\mu =0$,we define $V=0$. 

Conversely, a set $\mathfrak{h}$ with (1),(2) and (3) is an ideal of $\mbQ H$.
\end{Thm}
\section{Preparations for our main theorem}

\begin{Lem}[Key lemma]\label{Key}
If $x_1,\cdots ,x_n \in H\setminus {\rm ker}\mu$, there exists $z \in H$ that satisfies $\langle x_1,z\rangle \neq 0,\cdots$, and $\langle x_n,z\rangle \neq 0$. 
\end{Lem}
\begin{proof}
We prove this by induction on $n$. 
It is clear in the case $n=1$. Consider the case $n>1$. 
Take $u\in H$ satisfying the condition for $i=1,\cdots ,n-1$. 
If $\langle x_n ,u \rangle \neq 0$, there is nothing to do. Suppose not. 
We can choose $v\in H$ such that $\langle x_n,v\rangle \neq 0$, since $x_n \not \in {\rm ker}\mu$. We shall prove that
\[
z  := u + ( 1 + |\langle x_1,u\rangle | + \cdots +|\langle x_{n-1},u\rangle |)v
\]
is desired one. We have
\[
\langle x_n,z\rangle = (1 + |\langle x_1,u\rangle |+ \cdots +|\langle x_{n-1},u\rangle |)\langle x_n,v\rangle \neq 0 .
\]
For $k<n$, $\langle x_k,z\rangle =\langle x_k,u\rangle \neq 0$ if $\langle x_k,v\rangle =0$. 

If not, $\langle x_k,z\rangle \neq 0$, because
\[
|\langle x_k,z\rangle |
\geq  ( 1 + |\langle x_1,u\rangle | + \cdots +|\langle x_{n-1},u\rangle |)|\langle x_k,v\rangle | - |\langle x_k,u\rangle |
> 0 .
\]
\end{proof}
Let $\mathfrak{h}$ be an ideal of $\mbQ H$. Then we have
\begin{Prop}[Decomposition of an ideal with respect to ${\rm ker}\mu$-degree]\label{Decomp}
\[
\mathfrak{h} = \bigoplus _{\bar{x} \in \bar{H}} (\mathfrak{h}\cap \mbQ \bar{x})
\]
\end{Prop}
\begin{proof}
It is clear that the sum $\Sigma  (\mathfrak{h}\cap \mbQ \bar{x})$ is a direct sum and $\mathfrak{h}$ includes $\oplus _{\bar{x} \in \bar{H}}(\mathfrak{h}\cap \mbQ \bar{x})$. 
Let $X\in \mathfrak{h}\setminus 0$. 
Since $X \in \mathfrak{h}\subset \mbQ H = \oplus _{\bar{x}\in \bar{H}} \mbQ \bar{x}$, 
there exist $n \geq 1$and $x_i\in H$ and $X_i\in (\mbQ \bar{x_i})\setminus 0$ for $i=1,\cdots ,n$ such that $\bar{x_i}\neq \bar{x_j}$ if $i\neq j$ and $X=X_1 +\cdots + X_n$. 
It suffices to show $X_i\in \mathfrak{h}$ for all $i=1,\cdots ,n$. 

{\bf Step 1}
 {\it If $\bar{x_i}\neq 0$ for all $i=1,\cdots ,n$, then $X_i\in \mathfrak{h}$ for all $i=1,\cdots ,n$. } 
 
We show this case by induction on $n$, the number of the non-zero components. 

{\bf Claim} 
{\it Suppose $n>1$. For $k = 1,\cdots ,n$,
 there exist $c_1,\cdots $, and $c_n \in \mbQ$ such that (1) $c_k\neq 0$, (2) there exists some $j=1,\cdots ,n$ with $c_j= 0$, and (3) $c_1 X_1 + \cdots + c_n X_n \in \mathfrak{h}$. }

\begin{proof}[Proof of Claim]
We may assume $k = 1$. 
First we consider the case $\langle x_2,x_1\rangle \neq 0$. 
\[
Y:= ad([-x_2])ad([x_2])(X)
 = \sum _{i=1}^n\langle x_2,x_i\rangle \langle -x_2,x_i+x_2\rangle X_i
\]
Since $\langle x_2,x_1\rangle \langle -x_2,x_1+x_2\rangle \neq 0$, $\langle x_2,x_2\rangle =0$, and $Y\in \mathfrak{h}$, this claim holds. 

Second we consider the other case, i.e., $\langle x_2,x_1\rangle =0$. 
Since $\bar{x_1}\neq \bar{0},\ \bar{x_2}\neq 0$, and $\bar{x_1}\neq \bar{x_2}$, by Lemma \ref{Key},
 we can choose $z\in H$ that satisfies $\langle x_1,z\rangle \neq 0,\ \langle x_2,z \rangle \neq 0$, and $\langle x_1-x_2,z \rangle \neq 0$. 
\begin{eqnarray*}
 Y&:=&ad([x_2-z])ad([z])ad([-z])ad([z-x_2])(X) \\
&=&\sum _{i=1}^n \langle z-x_2,x_i \rangle \langle -z,x_i+z-x_2 \rangle \langle z,x_i-x_2 \rangle \langle x_2-z,x_i-x_2+z\rangle X_i
\end{eqnarray*}
Since $ \langle z-x_2,x_1 \rangle \langle -z,x_1+z-x_2 \rangle \langle z,x_1-x_2 \rangle \langle x_2-z,x_1-x_2+z\rangle \neq 0$, $\langle z,x_2-x_2\rangle =0$, and $Y\in \mathfrak{h}$, this claim holds. 
\end{proof}
\begin{proof}[Proof of Step 1]
Induction on $n$. 
If $n=1$, $X_1=X\in \mathfrak{h}$. 
Suppose $n>1$. 
By the inductive assumption and the assertions (2) and (3) of Claim, we have $c_k X_k \in \mathfrak{h}$. 
By the assertion (1) of Claim, we have $X_k\in \mathfrak{h}$. 

This proves Step 1. 
\end{proof}

{\bf Step 2}
{\it If $\bar{x_i}=\bar{0}$ for some $i=1,\cdots ,n$, then $X_j \in \mathfrak{h}$ for all $j=1,\cdots ,n$. }

\begin{proof}[Proof of Step 2]

The index $i$ with $\bar{x_i}=\bar{0}$ is unique since $\bar{x_i}\neq \bar{0}$ if $j\neq i$. 
We can assume $i=1$. 
If $n=1$, we have $X_1=X\in \mathfrak{h}$. 
Suppose $n>1$. 
Since $\bar{x_2}\neq \bar{0},\cdots$, and $\bar{x_n}\neq \bar{0}$, by Lemma \ref{Key}, there exists $z \in H$ such that $\langle x_2,z\rangle \neq 0,\cdots$, and $\langle x_n,z\rangle \neq 0$. 
\[
Y := ad([-z])ad([z])(X)
=\sum _{j=1}^n \langle z,x_j\rangle \langle -z,x_j+z\rangle X_j
\]
Since $\langle z,x_1\rangle =0$, and Step 1, we have
$\langle z,x_j\rangle \langle -z,x_j+z\rangle X_j \in \mathfrak{h}$ for all $i=1,\cdots ,n$. Since $\langle z,x_j \rangle \langle -z,x_j+z\rangle \neq 0$, we obtain $X_j \in \mathfrak{h}$ for all $i=2,\cdots ,n$. 
Moreover, $X_1=X-X_2-\cdots -X_n \in \mathfrak{h}$. 

This proves Step 2.
\end{proof}
Step 1 and Step 2 complete Proposition \ref{Decomp}. 
\end{proof}
\begin{Prop}[Homogeneity of an ideal]\label{Homog}
For $x,y \in H\setminus {\rm ker}\mu$,
\[
T(-x)(\mathfrak{h}\cap \mbQ \bar{x})
 = T(-y)(\mathfrak{h}\cap \mbQ \bar{y})
\]
\end{Prop}
\begin{proof}
Let $X^{\prime}\in T(-x)(\mathfrak{h}\cap \mbQ \bar{x})$.
 Then there exists $X\in \mathfrak{h}\cap \mbQ \bar{x}$ with $X^{\prime}=T(-x)(X)$. 
Since $X^{\prime} = T(-y)(T(y-x)(X))$ and ${\rm ker}\mu$-deg$T(y-x)(X)=\overline{y-x}+\bar{x}=\bar{y}$, it is sufficient to prove $T(y-x)(X)\in \mathfrak{h}$. 

By $x,y \in H\setminus {\rm ker}\mu$ and Lemma \ref{Key},
 we can choose $z\in H$ with $\langle x,z\rangle \neq 0$ and $\langle y,z\rangle \neq 0$. 
 We obtain $T(y-x)(X)\in \mathfrak{h}$ because
 \[
ad([-z])ad([y])ad([-x])ad([z])(X)
=(\langle z,x\rangle \langle y,z\rangle )^2T(y-x)(X).
\]
\end{proof}
\begin{Cor}[${\rm ker}\mu$-stability of an ideal]\label{Stab}
For $x\in H\setminus {\rm ker}\mu$ and $v\in {\rm ker}\mu$,
\[
T(v)(T(-x)(\mathfrak{h}\cap \mbQ \bar{x}))
=T(-x)(\mathfrak{h}\cap \mbQ \bar{x})
\qquad
[ \ T(v)(\mathfrak{h}\cap \mbQ \bar{x})
=\mathfrak{h}\cap \mbQ \bar{x} \ ]
\]
\end{Cor}
\begin{proof}
We apply Proposition \ref{Homog} to $x$ and $x-v$. Then
\[
T(-x)(\mathfrak{h}\cap \mbQ \bar{x})
=T(-(x-v))(\mathfrak{h}\cap \mbQ \overline{x-v})
=T(v)(T(-x)(\mathfrak{h}\cap \mbQ \bar{x}))
\]
\end{proof}
\section{Proof of Theorem \ref{The ideals}}
\begin{proof}
{\bf Existence:} When $\mu =0$, $\bar{0}=H$. So we can define $V_0 =\mathfrak{h}$. 
Assume $\mu \neq 0$.
Then we can choose $x_0 \in H\setminus {\rm ker}\mu$.

By Proposition \ref{Decomp}, we have
\[
\mathfrak{h}= \bigoplus _{\bar{x}\in \bar{H}}(\mathfrak{h}\cap \mbQ \bar{x}).
\]
Let $V_0:=\mathfrak{h}\cap \mbQ \bar{0}$ and $V:= T(-x_0)(\mathfrak{h}\cap \mbQ\bar{x_0})$.
By Corollary \ref{Stab}, $V$ has the ${\rm ker}\mu$-stability. 
And for all $y \in H$, we have 
\[
\mathfrak{h}\cap \mbQ \bar{y}
= T(y)(T(-y)(\mathfrak{h}\cap \mbQ \bar{y}))
= T(y)(T(-x_0)(\mathfrak{h}\cap \mbQ \bar{x_0}))
= T(y)(V)
\]
by Proposition \ref{Homog}. So we obtain
\[
\mathfrak{h} = V_0 \oplus \bigoplus _{\bar{x} \in \bar{H}\setminus \bar{0}}T(\bar{x})(V).
\]
{\bf Uniqueness:}
We assume that $(V_0,V)$ and $(W_0,W)$ satisfy (1), (2), and
\[
V_0 \oplus \bigoplus _{\bar{x} \in \bar{H}\setminus \bar{0}}T(\bar{x})(V)
= W_0 \oplus \bigoplus _{\bar{x} \in \bar{H}\setminus \bar{0}}T(\bar{x})(W).
\]
Then we obtain
\begin{eqnarray*}
V_0
&=&
\mbQ \bar{0} \cap  (V_0 \oplus \bigoplus _{\bar{x} \in \bar{H}\setminus \bar{0}}T(\bar{x})(V)) \\
&=&
\mbQ \bar{0} \cap  (W_0 \oplus \bigoplus _{\bar{x} \in \bar{H}\setminus \bar{0}}T(\bar{x})(W)) 
= W_0.
\end{eqnarray*}
If $\mu =0$,$V=0=W$. Suppose not.Take $\bar{x_0}\in \bar{H}\setminus \bar{0}$.
We obtain $V= W$ because
\begin{eqnarray*}
T(\bar{x_0})(V)
&=&
\mbQ \bar{x_0} \cap ( V_0 \oplus \bigoplus _{\bar{x} \in \bar{H}\setminus \bar{0}}T(\bar{x})(V)) \\
&=&
\mbQ \bar{x_0} \cap ( W_0 \oplus \bigoplus _{\bar{x} \in \bar{H}\setminus \bar{0}}T(\bar{x})(W)) 
=
T(\bar{x_0})(W).
\end{eqnarray*}
{\bf Converse:}
Assume that $(V_0,V)$ satisfies (1) and (2), and let $\mathfrak{h}$ be (3).
It is clear that $\mathfrak{h}$ is a subspace. 

For $X\in V_0$ and $Y \in \mbQ H$, we have $[X,Y]=0 \in \mathfrak{h}$. 

For $v \in V$, $x \in H\setminus {\rm ker}\mu$, and $y \in H$, we define
\[
z:= [T(x)(v),[y]]=\langle x,y\rangle T(x+y)(v).
\]
When $\overline{x+y} \neq 0$, $z\in T(\overline{x+y})(V)\subset \mathfrak{h}$. If not, $z = 0 \in \mathfrak{h}$ because
\[
0 = \langle x+y,y\rangle = \langle x,y\rangle .
\]
Hence, we obtain $[\mathfrak{h},\mbQ H] \subset \mathfrak{h}$.
\end{proof}

\section{Corollaries}
\begin{Prop}
Let $\mathfrak{h}$ be an ideal of $\mbQ H$, and $(V_0,V)$ the pair of $\mathfrak{h}$ in Theorem \ref{The ideals}. 
Then, $\mathfrak{h}$ is abelian if and only if $V=0$. 
\end{Prop}
\begin{proof}
Suppose $V=0$. 
Then, since $\mathfrak{h}=V_0\subset \mbQ {\rm {\rm ker}}\mu$, $\mathfrak{h}$ is abelian. 

Conversely, suppose $V\neq 0$. 
Then, since $\mu \neq 0$, there exist $x,y\in H$ with $\langle x,y\rangle \neq 0$. 
Take $v=\sum _{i=1}^n c_i[v_i]\in V\setminus 0$. 
Then, $\mathfrak{h}$ is not abelian, since $T(x)(v),T(y)(v)\in \mathfrak{h}$ and
\[
[T(x)(v),T(y)(v)]=\langle x,y\rangle \sum _{i=1}^n\sum _{j=1}^nc_ic_j[x+y+v_i+v_j]\neq 0.
\]
\end{proof}
We can calculate the center $\mathfrak{z}(\mbQ H)$, the derived subalgebra $[\mbQ H,\mbQ H]$ and the descent sequences $\mbQ H^{(m)}$ and $\mbQ H_{(m)} (m>0)$; 
$
\mathfrak{z}(\mbQ H)=\mbQ ({\rm {\rm ker}}\mu )
$, $
[\mbQ H,\mbQ H]=\mbQ (H\setminus {\rm {\rm ker}}\mu )
$, and $
\mbQ H^{(m)} =\mbQ H_{(m)}=\mbQ (H\setminus {\rm {\rm ker}}\mu )
$ 
where we define 
$
\mbQ H^{(1)}=\mbQ H_{(1)}=[\mbQ H,\mbQ H]
$, $
\mbQ H^{(m)}=[\mbQ H^{(m-1)},\mbQ H^{(m-1)}]
$, and $
\mbQ H_{(m)}=[\mbQ H,\mbQ H_{(m-1)}].
$
\begin{Def}
Let $M_1$, $M_2$ and $N$ be modules over a same commutative ring. 
A bilinear map $f:M_1\times M_2 \rightarrow N$ is said to be {\it non-degenerate} if $f$ satisfies the following two conditions; 

(1) If $f(x_1,x_2)=0$ for all $x_2 \in M_2$, then $x_1=0$. 

(2) If $f(x_1,x_2)=0$ for all $x_1 \in M_1$, then $x_2=0$. 
\end{Def}
\begin{ex}
Let $\Sigma$ be a surface with $\# \pi _0(\partial \Sigma )\leq 1$. 
We consider $H=H_1(\Sigma ,\mbZ )$ and the intersection form $\langle -,-\rangle$ on $H$. 
Then, $H$ is a free $\mbZ$-module, and $\langle -,-\rangle$ is a non-degenerate alternating $\mbZ$-bilinear form. 
\end{ex}
\begin{Cor}
If $\langle -,-\rangle$ is non-degenerate and $H\neq 0$, All the ideals are
\[
0 ,\quad \mbQ [0] ,
 \quad \mbQ (H\setminus 0) , \quad and \quad \mbQ H.
\]
\end{Cor}
\begin{proof}
Since $\langle -,-\rangle$ is non-degenerate, ${\rm ker}\mu = 0$. 
So all the subspace of $\mbQ \bar{0}$ are $0$ and $\mbQ [0]$.
And these have the ${\rm ker}\mu$-stability. 

Since $T(x)(\mbQ [0])=\mbQ [x]$, we obtain this Corollary. 
\end{proof}

\begin{rem}
We can define the homological Goldman Lie algebra $RH$ over an arbitrary commutative ring $R$ instead of $\mbQ$.
 And we can prove all the propositions in this paper if the coefficients of homological Goldman Lie algebras are in a commutative field of characteristic $0$.
\end{rem}


\end{document}